\documentclass[12pt]{amsart}
\usepackage{amssymb,latexsym,amsmath,amsthm,enumitem,mathrsfs,geometry}
\usepackage{fullpage}
\usepackage{fancyhdr}
\usepackage{xcolor}
\usepackage{bbm}
\usepackage{stmaryrd}
\usepackage{mathrsfs}
\usepackage{hyperref}
\usepackage{lineno}
\usepackage{diagbox}
\usepackage{array}
\usepackage{tikz}
\usetikzlibrary{arrows}
\usetikzlibrary{positioning}
\usepackage{float}
\usepackage{tabularx}  
\usepackage{array}     
\usepackage{comment}
\usepackage{mathtools}
\usepackage{cleveref}
\usepackage{tikz}
\usepackage{tikz-3dplot}
\usepackage{listings}
\usepackage{xcolor}
\usepackage{textcomp}
\usepackage[T1]{fontenc}
\usepackage{lmodern}
\usepackage{amsmath}    
\usepackage{listings}   
\definecolor{mygreen}{rgb}{0,0.6,0}

\usepackage[T1]{fontenc}
\usepackage{amsmath}
\usepackage{xcolor}
\usepackage{listings}
\usepackage{geometry}
\definecolor{codegreen}{rgb}{0,0.6,0}
\definecolor{codegray}{rgb}{0.5,0.5,0.5}
\definecolor{codepurple}{rgb}{0.58,0,0.82}
\definecolor{backcolour}{rgb}{0.95,0.95,0.95}
\definecolor{keywordblue}{rgb}{0.13, 0.13, 1}

\lstdefinestyle{mystyle}{
    backgroundcolor=\color{backcolour},   
    commentstyle=\color{codegreen},
    keywordstyle=\color{keywordblue}\bfseries,
    numberstyle=\tiny\color{codegray},
    stringstyle=\color{codepurple},
    basicstyle=\ttfamily\small,
    breakatwhitespace=false,         
    breaklines=true,                 
    captionpos=b,                    
    keepspaces=true,                 
    numbers=left,                    
    numbersep=5pt,                  
    showspaces=false,                
    showstringspaces=false,
    showtabs=false,                  
    tabsize=2,
    language=Mathematica,
    frame=single, 
    rulecolor=\color{black!30},
    mathescape=true, 
    morekeywords={Reduce, Reals} 
}
\usepackage{tikz}
\usepackage{pgfplots}
\pgfplotsset{compat=1.18}
\usepackage{tikz-3dplot}
\usepackage{caption}
\usepackage{amsmath}
\usetikzlibrary{arrows.meta} 

\newtheorem{theorem}{Theorem}[section]
\newtheorem*{theorem*}{Theorem}

\newtheorem*{remark*}{Remark}

\newtheorem{corollary}{Corollary}[section]
\newtheorem{proposition}[theorem]{Proposition}

\newtheorem{definition}{Definition}[section]

\newcommand\JP[1]{\langle#1\rangle}
\newcommand\norm[1]{\left\|#1\right\|}

\newcommand\Inpro[3]{\langle #1,#2\rangle_{#3}}
\newcommand\RE{\text{Re}}

\newcommand{\R}{\mathbb{R}}
\newcommand{\C}{\mathbb{C}}
\newcommand{\Z}{\mathbb{Z}}
\def\d{\mathrm{d}}

\begin{document}
	
	\numberwithin{equation}{section}
	
	\title{Strichartz estimate for discrete Schr\"odinger equation on layered King's grid}
	
	\author[Wan]{Zhiqiang Wan}
	\address{School of Mathematical Sciences, University of Science and Technology of China, Hefei}
	\email{ZhiQiang\_Wan576@mail.ustc.edu.cn}
	\author[Zhang]{Heng Zhang}
	\address{School of Mathematical Sciences, University of Science and Technology of China, Hefei}
	\email{hengz@mail.ustc.edu.cn}

	\keywords{}
	\subjclass[2020]{}
	\thanks{}

	\date{\today}

	\begin{abstract}
		We establish the sharp \( l^1 \to l^{\infty} \) decay estimate for the discrete Schr\"odinger equation (DS) on the Layered King's Grid (LKG), with a dispersive decay rate of \( \langle t \rangle^{-13/12} \), which is faster than that for $3$-dimensional lattice (\( \langle t \rangle^{-1} \), see \cite{SK05}). This decay estimate enables us to derive the corresponding Strichartz estimate via the standard Keel--Tao argument. Our approach relies on using techniques from Newton polyhedra to analyze singularities.
		
	\end{abstract}

	\maketitle
	
	\section{Introduction } \label{intro}
	The Schr\"odinger equations, particularly those on discrete spaces such as certain graphs, are fundamental dispersive equations that have a wide range of applications in physical models\cite{BCFMMI02, CLS03, FW12, FW20, KAT01, LFO06, MPAES99, OOSB01}. Let $G=(V, E)$ be a simple connected graph. The discrete Laplacian on $G$ is defined as:
	\begin{equation*}
		\Delta_{\mathrm{disc}}f(x) := \sum_{y\sim x} (f(y) - f(x)),
	\end{equation*}
	where $y\sim x$ indicates that $y$ and $x$ are adjacent vertices in $G$. 
	We are concerned with the dispersive  estimate for the discrete Schr\"odinger equation on $G$, that is, 
	\begin{equation}\label{eq:DS}
		\begin{cases}
			\left(\partial_{t}-i\Delta_{\mathrm{disc}}\right)u\left(x,t\right)=0&\mathrm{for}\;\left(x,t\right)\in V\times\R_{>0},\\
			u\left(x,0\right)=u_{0}\left(x\right)&\mathrm{for}\;\;x\in V.
		\end{cases}
	\end{equation}
	where $u:{V}\times\R_{\geq0}\to\C$,  and $u_{0}$ is the initial data.
	Demonstrating dispersion in discrete spaces often involves establishing a decay estimate for the $ l^\infty $ norm of the solution at time $ t $, expressed in terms of some negative power of $ t $ and the $ l^1 $ norm of the initial data, namely $ l^1 \to  l^\infty$ estimate(or $ L^1 \to  L^\infty$ estimate in the continuous space). It has the following form:
	\begin{align}\label{Strichartz estimate}
		\|u(x,t)\|_{l^\infty(V)}\le C\left\langle t\right\rangle^{-\sigma}\|u_0\|_{l^1(V)},
	\end{align}
where $\left\langle t\right\rangle$ denotes the Japanese bracket, defined as $\left\langle t\right\rangle:=\sqrt{1+t^{2}}$, and $C$ is a constant that is independent of $t$. Combining this decay estimate with an abstract functional analysis argument, i.e., the \( TT^* \) argument yield a variety of inequalities involving space-time Lebesgue norms, called Strichartz estimate, which was first introduced in \cite{S77}.
	
In the Euclidean space $\R^d$, the $L^1 \to L^{\infty}$ estimate is well-known to hold with $\sigma = d/2$, see, for example, \cite{S77, T06}. However, in the $d$-dimensional lattice, the sharp $l^1 \to l^{\infty}$ estimate of equation \eqref{eq:DS} holds with $\sigma = d/3$, which is slower than the corresponding estimate in $\R^d$, as shown in \cite[Proposition 1]{SK05}. The similar dispersive estimates and corresponding applications for other related discrete equations can be found in \cite{B13, CT09, EKT15, CI21, KPS09, MP12, PS08}.  The study of the Schr\"odinger equation on Cayley graphs is a particularly fruitful area of research, as the inherent group symmetry can be exploited by tools from harmonic analysis to decompose the Schrödinger equation, thereby revealing deep connections among the algebraic properties of the group, the geometry of the graph, and the operator's spectrum (see, e.g., \cite{CT09, EKT15, SK05, PS08, GHJZ25}).
	
	A natural line of inquiry is to examine the $l^1 \to l^{\infty}$ dispersive decay on other Cayley graphs with vertex set  $\mathbb{Z}^d$. In a recent work, Ge, Hua, Jia, and Zhou \cite{GHJZ25} investigated this for the hexagonal lattice $T_H$, viewed as the Cayley graph of $\mathbb{Z}^2$ with the generating set ${\pm(1,0), \pm(0,1), \pm(1,1)}$. Their analysis revealed that the decay rate is $\langle t \rangle^{-3/4}$, a notable acceleration over the classical $\langle t \rangle^{-2/3}$ rate for the standard square lattice.

	We consider the Strichartz estimate on Layered King's Grid  (LKG), which has  meaningful geometric and physical interpretations, defined as the Cayley graph of $\Z^3$ and the edge set $E$ given by $E = \left\{ \{u, v\} \mid u - v \in \{(\pm 1,0,0),(0,\pm 1,0), (0,0,\pm1),  (\pm 1,\pm 1,0)\} \right\}$.  LKG is an ideal model for anisotropic cellular automata. It combines complex 2D Moore-like interactions within layers with simpler 1D von Neumann-like dynamics between them. This structure is perfect for modeling layered phenomena, like crystal growth, offering a more realistic framework than purely isotropic models. For more about cellular automata, we refer to \cite{CC10, TM95, W00}.
	Our research offers a adaptable framework for characterizing continuous-time quantum walk on anisotropic cellular automata, see \cite{S08}.

 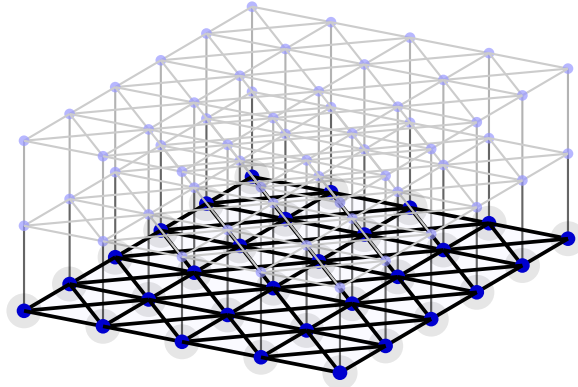
\begin{figure}[h!]
    \centering
    \tdplotsetmaincoords{70}{120}
    
    \begin{tikzpicture}[scale=1.2, tdplot_main_coords]
        
        \def\Xmax{5}
        \def\Ymax{4}
        \def\Zmax{2}
        
        \fill[blue!10, opacity=0.2] 
            (0,0,0) -- (\Xmax,0,0) -- (\Xmax,\Ymax,0) -- (0,\Ymax,0) -- cycle;
        
        \foreach \x in {0,...,\Xmax} {
            \foreach \y in {0,...,\Ymax} {
                \foreach \z in {0,...,\the\numexpr\Zmax-1\relax} {
                    \pgfmathsetmacro{\opacity}{0.6 - 0.3*\z}
                    \draw[black, thick, opacity=\opacity] (\x,\y,\z) -- (\x,\y,\z+1);
                }
            }
        }
        
        \foreach \z in {0,...,\Zmax} {
            \foreach \x in {0,...,\Xmax} {
                \foreach \y in {0,...,\Ymax} {
                    \ifnum\z=0
                        \node[circle, fill=blue!90!black, inner sep=0pt, minimum size=5.5pt] 
                            at (\x,\y,\z) {};
                    \else
                        \node[circle, fill=blue!40, inner sep=0pt, minimum size=4pt, opacity=0.7] 
                            at (\x,\y,\z) {};
                    \fi
                    
                    \ifnum\x<\Xmax
                        \ifnum\z=0
                            \draw[black, line width=1.4pt] (\x,\y,\z) -- (\x+1, \y, \z);
                        \else
                            \draw[gray!40, line width=0.8pt] (\x,\y,\z) -- (\x+1, \y, \z);
                        \fi
                    \fi
                    \ifnum\y<\Ymax
                        \ifnum\z=0
                            \draw[black, line width=1.4pt] (\x,\y,\z) -- (\x, \y+1, \z);
                        \else
                            \draw[gray!40, line width=0.8pt] (\x,\y,\z) -- (\x, \y+1, \z);
                        \fi
                    \fi
                    \ifnum\x<\Xmax \ifnum\y<\Ymax
                        \ifnum\z=0
                            \draw[black, line width=1.4pt] (\x,\y,\z) -- (\x+1, \y+1, \z);
                        \else
                            \draw[gray!40, line width=0.8pt] (\x,\y,\z) -- (\x+1, \y+1, \z);
                        \fi
                    \fi\fi
                    \ifnum\x<\Xmax \ifnum\y>0
                        \ifnum\z=0
                            \draw[black, line width=1.4pt] (\x,\y,\z) -- (\x+1, \y-1, \z);
                        \else
                            \draw[gray!40, line width=0.8pt] (\x,\y,\z) -- (\x+1, \y-1, \z);
                        \fi
                    \fi\fi
                }
            }
        }
        
        \foreach \x in {0,...,\Xmax} {
            \foreach \y in {0,...,\Ymax} {
                \fill[black, opacity=0.1] (\x,\y,0) circle (5.5pt);
            }
        }
        
    \end{tikzpicture}
    
    \caption{Layered Kings Grid on $\R^3$}
    \label{fig:professional-3d-graph}
\end{figure}
In our setting, we use the standard Fourier transform:   
\begin{equation*}
	\widehat{f}\left(x\right):=\frac{1}{\left(2\pi\right)^{d/2}}\sum_{\mathbf{n}\in\Z^{d}}f\left(\mathbf{n}\right)e^{-ix\cdot \mathbf{n}}\;\mathrm{for}\;x\in\left[0,2\pi\right]^{d},
\end{equation*}
where the sequence $f:\Z^{d}\to \C$ belongs to $l^{2}\left(\Z^{d}\right)$. The inverse Fourier transform is given by 
\begin{equation*}
	\check{f}\left(\mathbf{n}\right):=\frac{1}{\left(2\pi\right)^{d/2}}\int_{\left[0,2\pi\right]^{d}}f\left(x\right)e^{ix\cdot \mathbf{n}}\,\d x\;\mathrm{for}\;\mathbf{n}\in\Z^{d},
\end{equation*}
where $f:\left[0,2\pi\right]^{d}\to\C$ belongs to $L^{2}\left(\left[0,2\pi\right]^{d}\right)$. 
	
Applying the Fourier transform on LKG for the equation \eqref{eq:DS}, we have that
    \begin{equation*}
         \begin{aligned} u\left(x,y,z,t\right)&=:e^{it\Delta_{\mathrm{disc}}}u_{0}\left(x,y,z\right)=u_{0}\ast\left(\psi_{t}\right)^{\vee}\left(x,y,z\right)\\
         &=\sum_{\mathbf{n}\in\Z^{3}}u_{0}\left(\mathbf{n}\right)\psi_{t}\left(\left(x,y,z\right)-\mathbf{n}\right)\;\;\mathrm{\;for\;any\;}\left(x,y,z\right)\in\Z^{3}.
		\end{aligned}
	\end{equation*}
Here $\psi_{t}$ is defined by
\begin{equation*}
    \begin{aligned}
	\psi_{t}\left(x,y,z\right)=&\frac{1}{\left(2\pi\right)^{3/2}}\exp\Big(-it\big(10-2\cos x-2\cos y-2\cos z-2\cos\left(x+y\right)-2\cos\left(x-y\right)\big)\Big)\\ 
    =&\frac{1}{\left(2\pi\right)^{3/2}}\exp\left(-it \left(\omega\left(x,y\right)+\omega_{1}\left(z\right)\right)\right)\;\;\mathrm{for \;any}\;\left(x,y,z\right)\in\Z^{3},
    \end{aligned}
\end{equation*}
where  $\omega_{1}\left(z\right)=2-2\cos z$, and 
\begin{equation}\label{eq:omega} \omega\left(x,y\right)=8-2\cos x-2\cos y-2\cos\left(x+y\right)-2\cos\left(x-y\right).
\end{equation}
The Hausdorff-Young inequality gives that
\begin{equation}\label{eq:HY}
    \left\| u\left(t\right)\right\|_{l^{\infty}\left(\mathbb{Z}^{3}\right)}\leq\left\| \left(\psi_{t}\right)^{\vee}\right\|_{l^{\infty}\left(\mathbb{Z}^{3}\right)}\cdot\left\| u_{0}\right\|_{l^{1}\left(\mathbb{Z}^{3}\right)},
	\end{equation}
and we hope to give a decay estimate of $\left\| \left(\psi_{t}\right)^{\vee}\right\|_{l^{\infty}\left(\mathbb{Z}^{3}\right)}$ about time variable $t$ in the sense of 
\begin{equation*}
    \left\| \left(\psi_{t}\right)^{\vee}\right\|_{l^{\infty}\left(\mathbb{Z}^{3}\right)}\lesssim\JP{t}^{-\sigma},
\end{equation*}
for some $\sigma>0$. Changing variables $\left(\xi^{\prime},\eta^{\prime},\zeta^{\prime}\right)=\left(\xi,\eta,\zeta\right)t$ and using the Fubini theorem, the oscillatory integral takes the form
	\begin{equation*}
		\begin{aligned}
			\left(\psi_{t}\right)^{\vee}\left(t\xi,t\eta,t\zeta\right)=\frac{1}{\left(2\pi\right)^{3}}\int_{\left[0,2\pi\right]^{2}}e^{-it\left(\omega\left(x,y\right)-\left(\xi x+\eta y\right)\right)}\d x\d y\int_{\left[0,2\pi\right]}e^{-it\left(\omega_{1}\left(z\right)-\zeta z\right)}\d z,
		\end{aligned}
	\end{equation*} 
where $\left(\xi^{\prime},\eta^{\prime},\zeta^{\prime}\right)\in\R^{3}$ is the velocity. 
As for the integral with respect to $z$, 
\begin{equation*}
	\left|\frac{\d^{3}}{\d z^{3}}\left(\omega_{1}\left(z\right)-z\zeta\right)\right|^{2}+\left|\frac{\d^{2}}{\d z^{2}}\left(\omega_{1}\left(z\right)-z\zeta\right)\right|^{2}=4
\end{equation*}
holds and by the van der Corput's lemma (see \cite[Chapter VIII]{SM93}) we derive
\begin{equation}\label{ineq:lattice case}
	\left|\int_{\left[0,2\pi\right]}e^{-it\left(\omega_{1}\left(z\right)-z\zeta\right)}\d z\right|\lesssim \JP{t}^{-1/3},
\end{equation}
and this estimate is sharp {\cite{SK05}}. We now reduce the problem to the uniform estimate of 
\begin{equation}\label{eq:main oscillatory integral}
	\begin{aligned}
		I\left(\xi,\eta;t\right)=&\int_{\left[0,2\pi\right]^{2}}\exp\left(-it\left(\omega\left(x,y\right)-\left(x\xi+y\eta\right)\right)\right)\d x\d y\\
		:=&\int_{\left[0,2\pi\right]^{2}}\exp\left(-it\Omega\left(x,y;\xi,\eta\right)\right)\d x\d y,
	\end{aligned}
\end{equation}
 and prove its sharp decay rate is $\JP{t}^{-3/4}$. Leveraging the detailed singularity analysis of the phase function $\Omega\left(x,y;\xi,\eta\right)$ presented in Section \ref{sec2}, we derive the desired estimate.

\begin{theorem}\label{thm:main}
    Let $u$ be a solution to \eqref{eq:DS}, we have the following estimate \begin{align}
        \|u(x,t)\|_{l^\infty(V)}\le C\left\langle t\right\rangle^{-13/12}\|u_0\|_{l^1(V)},
	\end{align}
 where $C$ is a constant
that is independent of $t$.
\end{theorem}

    The dispersive decay rate of LKG is faster than that for that for $3$-dimensional lattice (\( \langle t \rangle^{-1} \), see \cite{SK05}). We explain this phenomenon using the same reasoning as in \cite{GHJZ25}, where the authors suggested this enhancement is attributed to two potential factors: the greater connectivity afforded by the additional generators, and a geometry whose geodesic balls provide a closer approximation to Euclidean balls. We remark that the estimate is sharp in the sense of combining with \eqref{ineq:lattice case} and \eqref{sharp}.

For the inhomogeneous discrete Schr\"odinger equation, 
\begin{equation}\label{equation: inhomogegeneous Schrodinger equation}
		\begin{cases}
			&\left(\partial_{t}-i\Delta_{disc}\right)u\left(x,t\right)=F\left(x,t\right)\;\;\mathrm{for}\;\left(x,t\right)\in\Z^{3}\times\R_{>0},\\
			&u\left(x,0\right)=u_{0}\left(x\right)\;\;\mathrm{for}\;x\in\Z^{3},
		\end{cases}
	\end{equation}
we have the following Strichartz estimate via standard Keel-Tao argument.
\begin{theorem}\label{thm:Strichartz}
Let $u\in C^1([0,+\infty);l^2(\mathbb{Z}^3))$ be a solution of the equation \eqref{equation: inhomogegeneous Schrodinger equation}, the Strichartz estimate holds with any two Strichartz $13/12$-admissible pairs $\left(q,r\right)$ and $\left(\tilde{q},\tilde{r}\right)$. That is 
\begin{equation}
	\norm{u}_{L^{q}_{t}l^{r}_{n}\left(\left(0,T\right)\times\Z^{3}\right)}\leq C\left(\norm{u_{0}}_{l^{2}\left(\Z^{3}\right)}+\norm{F}_{L^{\tilde{q}^{\prime}}_{t}l^{\tilde{r}^{\prime}}_{n}\left(\left(0,T\right)\times\Z^{3}\right)}\right),
\end{equation}
where $T\in (0, \infty]$.
\end{theorem}
The definition of \textit{mixed space-time Lebesgue norm} and \textit{$\sigma$-admissible}, we will recall in Subsection \ref{subsec3.1}. Further, in Subsection \ref{subsec3.2}, we will establish the global well-posedness for small initial data to discrete nonlinear Schr\"odinger equation on LKG.

By leveraging the Cartesian product structure of the LKG, we perform a separation of variables, which reduces the problem to deriving the required estimate solely for the King graph factor. A natural direction for future research is to identify other Cayley graphs of $\mathbb{Z}^d$ that lack a Cartesian product decomposition yet exhibit a decay rate faster than that of the d-dimensional lattice. We note that the main difficulty lies in higher dimensions; general results analogous to Theorem \ref{theorem: linear perturbation} remain an open problem in analysis and have thus far only been obtained for multidimensional oscillatory integrals with specific phase functions.

\section{Determining singularities of the phase function}\label{sec2}





In this section we will deal with a more general type of oscillatory integral
\begin{equation}\label{def: General oscillatory integral on the torus}
    I\left(\xi,\eta;t;\chi\right)=\int_{\left[0,2\pi\right]^{2}}e^{-it\Omega\left(x,y;\xi,\eta\right)}\chi\left(x,y\right)\d x\d y
\end{equation}
and give an asymptotic estimate about it when $t$ goes to infinity. Here, $\left(\xi,\eta\right)$ is a fixed vector in $\R^{2}$ which stands for the velocity and the function $\chi$ is a given smooth test function defined on $\left[0,2\pi\right]^{2}$ with periodic boundary conditions, i.e., $\chi$ has a natural periodic extension to the whole plane. In equation \eqref{eq:main oscillatory integral}, we take $\chi=\mathbbm{1}_{\left[0,2\pi\right]^{2}}$. 

We choose a smooth bump function $\mathbf{a}:\R^{2}\to\left[0,1\right]$ such that
\begin{itemize}
    \item $\min_{\left(x,y\right)\in\left[0,2\pi\right]^{2}}\mathbf{a}\left(x,y\right)\geq c_{0}$ for some positive number $c_{0}$,
    \item The support of $\mathbf{a}$ is contained in $\left[-4\pi,4\pi\right]^{2}$.
\end{itemize}
Then by the bounded overlapping property we can naturally assume that 
\begin{equation*}
    \sum_{\left(n,m\right)\in\Z^{2}}\mathbf{a}\left(x+2\pi n,y+2\pi m\right)\equiv 1.
\end{equation*}
Hence the oscillatory integral \eqref{def: General oscillatory integral on the torus} becomes
\begin{equation*}
    \begin{aligned}
        I\left(\xi,\eta;t;\chi\right)=&\sum_{\left(n,m\right)\in\Z^{2}}\int_{\left[0,2\pi\right]^{2}}e^{-it\Omega\left(x,y;\xi,\eta\right)}\mathbf{a}\left(x+2\pi n,y+2\pi m\right)\chi\left(x,y\right)\d x\d y\\
        =&\sum_{\left(n,m\right)\in\Z^{2}}\int_{\left[2\pi n,2\pi\left(n+1\right)\right]\times\left[2\pi m,2\pi\left(m+1\right)\right]}e^{-it\Omega\left(x,y;\xi,\eta\right)}\mathbf{a}\left(x,y\right)\chi\left(x,y\right)\d x\d y\\
        =&\int_{\R^{2}}e^{-it\Omega\left(x,y;\xi,\eta\right)}\mathbf{a}\left(x,y\right)\chi\left(x,y\right)\d x\d y,
    \end{aligned}
\end{equation*}
and we denote $\widetilde{I}\left(\xi,\eta;t;\chi\right)$ be
\begin{equation}\label{def: general oscillatory integral on the plane}
    \int_{\R^{2}}e^{-it\Omega\left(x,y;\xi,\eta\right)}\mathbf{a}\left(x,y\right)\chi\left(x,y\right)\d x\d y,
\end{equation}
where $\mathbf{a}(x,y)\chi(x,y)$ is a smooth compactly supported function defined on $\R^{2}$. In the definition of $\widetilde{I}\left(\xi,\eta;t;\chi\right)$, the function $\chi$ can be chosen as a periodic function, or just a smooth bump function.

When the velocity $\left(\xi,\eta\right)\in \R^{2}\setminus B_{7}\left(0\right)$, the phase function $\Omega\left(\cdot,\cdot;\xi,\eta\right)$ has no critical points in $\left[0,2\pi\right]^{2}$, since the function $\omega_{1}$ satisfies $\left|\nabla\omega_{1}\right|\leq 6$, and $\widetilde{I}\left(\xi,\eta;t;\chi\right)$ decays faster than any polynomial in $t$. If $\left(\xi,\eta\right)\in B_{7}\left(0\right)$, the phase function might has critical points in the support of $\chi$, which slow down the decay rate. In the classical theory of oscillatory integrals, we mainly focus on  critical points, here we call a point $x$ is a \textit{degenerate} point if the Hessian matrix of phase function at $x$ is degenerate, else we say $x$ is a \textit{non-degenerate} point. 	
Recall the phase function $\Omega\left(x,y;\xi,\eta\right)$ is given by
\begin{equation*}
	\omega\left(x,y\right)-\left(x\xi+y\eta\right).
\end{equation*}
 For fixed $\left(x_{0},y_{0}\right)\in\left[0,2\pi\right]^{2}$, there is an unique choice of the velocity $\left(\xi_{0},\eta_{0}\right)\in\R^{2}$ such that $\left(x_{0},y_{0}\right)$ is a critical point of $\Omega\left(x,y;\xi_{0},\eta_{0}\right)$, indeed, one can take $\left(\xi_{0},\eta_{0}\right)=\nabla_{\left(x_{0},y_{0}\right)}\omega$. The order of degeneracy of the phase function at that point is determined by the high order derivatives of $\omega$, and at that point of view, we give a partition of $\left[0,2\pi\right]^{2}$ :
\begin{equation*}
	\left[0,2\pi\right]^{2}=\Sigma_{1}\cup\Sigma_{2}\cup\Sigma_{3},
\end{equation*}
and $\Sigma_{i}$ are defined as 
\begin{equation}\label{def: Sigmai}
	\begin{aligned}
		\Sigma_{1}&=\left\{\left(x_0,y_0\right)\in\left[0,2\pi\right]^{2}:\det\left(\mathrm{Hess}_{\left(x_0,y_0\right)}\omega\right)\neq0\right\},\\
		\Sigma_{2}&=\left\{\left(x_0,y_0\right)\in\left[0,2\pi\right]^{2}:\left(x_0,y_0\right)\;\mathrm{is\;an}\;A_{2}\;\mathrm{type\;singularity\;of}\;\omega\left(x,y\right)-\left(x,y\right)\cdot\nabla_{\left(x_0,y_0\right)}\omega\right\},\\
		\Sigma_{3}&=\left\{\left(x_0,y_0\right)\in\left[0,2\pi\right]^{2}:\left(x_0,y_0\right)\;\mathrm{is\;an}\;A_{3}\;\mathrm{type\;singularity\;of}\;\omega\left(x,y\right)-\left(x,y\right)\cdot\nabla_{\left(x_0,y_0\right)}\omega\right\}.\\
	\end{aligned}
\end{equation}
Here, we remark that in our setting, the types of singularity are exactly $A_k$, $k\leq 3$ (See Propositon \ref{prop: velocity space}). The definition of $A_k$ type singularities, one can refer to \cite{A1, A2, A3}.

Using this partition, we divide the velocity space as follows
\begin{equation*}
    \R^{2}=V_{0}\cup V_{1}\cup V_{2}\cup V_{3},
\end{equation*}
where $V_{i}$ are given by
\begin{equation}\label{def: Vi}
	\begin{aligned}
        V_{0}&=\left\{v\in\R^{2}:\mathrm{for\;all}\;\left(x_{0},y_{0}\right)\in\left[0,2\pi\right]^{2},\;v\neq\nabla_{\left(x_{0},y_{0}\right)}\omega\right\},\\
		V_{3}&=\left\{v\in\R^{2}:\mathrm{there\;exists\;some}\;\left(x_{0},y_{0}\right)\in \Sigma_{3},\;\mathrm{such\;that}\;v=\nabla_{\left(x_{0},y_{0}\right)}\omega\right\},\\
        V_{2}&=\left\{v\in\R^{2}\setminus V_{3}:\mathrm{there\;exists\;some}\;\left(x_{0},y_{0}\right)\in \Sigma_{2},\;\mathrm{such\;that}\;v=\nabla_{\left(x_{0},y_{0}\right)}\omega\right\},\\
        V_{1}&=\left\{v\in\R^{2}\setminus \left(V_{2}\cup V_{3}\right):\mathrm{there\;exists\;some}\;\left(x_{0},y_{0}\right)\in \Sigma_{1},\;\mathrm{such\;that}\;v=\nabla_{\left(x_{0},y_{0}\right)}\omega\right\}.\\
	\end{aligned}
\end{equation}
$\{V_{i}\}_{i=0}^{3}$ actually divides the decay rates into several regions, see our Theorem \ref{theorem: POU+decay}.

In the following, we introduce some basics of Newton polyhedra, see, for example, \cite{A3, V76, PS97,IM16, IKM10}. Let $\varphi$ be an analytic function defined on a sufficiently small neighborhood $U$ of the origin, such that $\varphi\left(0,0\right)=0$ and $\nabla\varphi\left(0,0\right)=0$, its associated Taylor series around the origin is given by
\begin{equation*}
    \varphi\left(x,y\right)=\sum_{n_{1},n_{2}=0}^{\infty}a_{n_{1},n_{2}}x^{n_{1}}y^{n_{2}}.
\end{equation*}
The set
\begin{equation*}
    \mathcal{T}\left(\varphi\right):=\left\{\left(n_{1},n_{2}\right)\in\mathbb{N}^{2}:a_{n_{1},n_{2}}\neq 0\right\}
\end{equation*}
is called the \textit{Taylor support} of $\varphi$ at $\left(0,0\right)$.  The function $\varphi$ is of \textit{finite type} under the assumptions $\varphi\left(0,0\right)=0$ and $\nabla\varphi\left(0,0\right)=0$ just means that $\mathcal{T}\left(\varphi\right)\neq \emptyset$. The \textit{Newton polyhedron} $\mathcal{N}\left(\varphi\right)$ of $\varphi$ at the origin is defined be the convex hull of the union of all the quadrants $\left(n_{1},n_{2}\right)+\R_{+}^{2}$ with $\left(n_{1},n_{2}\right)\in\mathcal{T}\left(\varphi\right)$, where $\R_{+}^{2}:=\left[0,\infty\right)^{2}$. The \textit{face} in a Newton polyhedron is an edge or a vertex. For a Newton polyhedron $\mathcal{N}\left(\varphi\right)$, the \textit{Newton distance} $d_{\varphi}$ is defined by
\begin{equation*}
    d_{\varphi}:=\inf\left\{d>0:\left(d,d\right)\in\mathcal{N}\left(\varphi\right)\right\}.
\end{equation*}
The \textit{principal face} $\pi\left(\varphi\right)$ of the Newton polyhedron of $\varphi$ is the face of minimal dimension containing the point $\left(d_{\varphi},d_{\varphi}\right)$. By a \textit{local coordinate system(at the origin)} we shall mean a smooth coordinate system defined near the origin which preserves 0. The \textit{height} of the smooth function $\varphi$ is defined by
\begin{equation*}
    h_{\varphi}:=\sup\left\{d_{y}\right\}
\end{equation*}
where the supremum is taken over all local coordinate systems $y=\left(y_{1},y_{2}\right)$ at the origin and where $d_{y}$ is the distance between the Newton polyhedron and the origin in the coordinate $y$. A given coordinate system $x$ is said to be \textit{adapted} to $\varphi$ if $h\left(\varphi\right)=d_{x}$. The \textit{Varchenko's exponent} \cite{V76} $\nu\left(\varphi\right)\in\{0,1\}$ as follows:

If there exists an adapted local coordinate system $y$ near the origin such that the principal face $\pi\left(\varphi^{a}\right)$ of $\varphi$, when expressed by the function $\varphi^{a}$ in the new coordinates, is a vertex, and if $h\left(\varphi\right)\geq 2$, then we put $\nu\left(\varphi\right):=1$; otherwise, we put $\nu\left(\varphi\right):=0$.

The following two propositions associated with Newton polyhedra will be used in proofs of our Propositions \ref{prop: velocity space} and \ref{prop: heights}.

\begin{proposition}[\cite{V76}, part 2 of Proposition 0.7]\label{proposition: A_{3}-type singularity}
		Assume that for a given series $f=\sum c_{n}y^{n}$, the point $\left(d\left(f\right),d\left(f\right)\right)$ lies on a closed compact face $\Gamma$ of the Newton's polyhedron. Let $a_{1}n_{1}+a_{2}n_{2}=m$ be the equation of the straight line on which $\Gamma$ lies, where $a_{1}$, $a_{2}$, and $m$ are integers and $a_{1}$ and $a_{2}$ are relatively prime. The the coordinate system $y$ is adapted if both numbers $a_{1}$ and $a_{2}$ are larger than 1.
\end{proposition}
\begin{proposition}[\cite{V76}, Proposition 0.8]\label{prop:V0.8}
    Assume that for a given series $f=\sum c_{n}y^{n}$, the point $\left(d\left(f\right),d\left(f\right)\right)$ lies on a closed compact face $\Gamma$ of the Newton's polyhedron. Let $a_{1}n_{1}+n_{2}=m$ be the equation of the straight line on which $\Gamma$ lies, where $a_{1}$ and $m$ are integers. Let 
    \begin{equation*}
        f_{\Gamma}\left(y\right)=\sum_{n\in\Gamma}c_{n}y^{n}\;\;\;\mathrm{and}\;\;\;P\left(y_{1}\right)=f_{\gamma}\left(y_{1},1\right).
    \end{equation*}
    If the polynomial $P$ does not have a real root of multiplicity larger than $m\left(1+a_{2}\right)^{-1}$, then $y$ is a coordinate system adapted to $f$.
\end{proposition}

The techniques from Newton polyhedra will be used to analyze singularities. We say a critical point is an $A_k$ type singularity, if it has normal form $x^{k+1}$. For more basics about singularities and normal forms, we refer to \cite{A1, A2, A3}.
In our setting, we derive Propositions \ref{prop: velocity space} and \ref{prop: heights}, which respectively identify the singularities and heights. Proposition \ref{prop: heights} yields $h<2$, and hence $\nu\equiv 0$. Their proofs are deferred to Section \ref{sec3}.

\begin{proposition}\label{prop: velocity space}
    For any $\left(x_{0},y_{0}\right)\in\left[0,2\pi\right]^{2}$, $\left(x_0,y_0\right)$ is an $A_{k}$, $k\leq 3$ type singularity of the phase function $\omega\left(x,y\right)-\left(x,y\right)\cdot\nabla_{\left(x_0,y_0\right)}\omega.$

    
\end{proposition}

\begin{proposition}\label{prop: heights}
    The height $h$ of the function $\omega\left(x,y\right)-\left(x,y\right)\cdot\nabla_{\left(x_0,y_0\right)}\omega$ at point $\left(x_0,y_0\right)$, is constant in each $\Sigma_{i}$ for $i=1,2,3$ and takes values as follows:
    \begin{itemize}
        \item For $\left(x_0,y_0\right)\in\Sigma_{1}$, $h=1$,
        \item For $\left(x_0,y_0\right)\in\Sigma_{2}$, $h=6/5$,
        \item For $\left(x_0,y_0\right)\in\Sigma_{3}$, $h=4/3$.
    \end{itemize}
\end{proposition}
The next Theorem \ref{theorem: linear perturbation} is quoted from \cite{IM11}, shows that the oscillatory integral is stable under the small linear perturbations of the phase function. 
\begin{theorem}[\cite{IM11}, Theorem 1.1]\label{theorem: linear perturbation}
	Let $\varphi$ be a smooth, real-valued phase function of finite type, defined near the origin, as before, and let $h:=h\left(\varphi\right)$, $\nu:=\nu\left(\varphi\right)$. Then there exist a neighborhood $\Omega\subset\R^{2}$ of the origin and a constant $C$ such that for every $\eta\in C^{\infty}_{0}\left(\Omega\right)$ the following estimate holds true for every $\xi\in\R^{3}$:
\begin{equation*}
	\left|\int_{\R^{2}}e^{i\left(\xi_{3}\varphi\left(x_{1},x_{2}\right)+\xi_{1}x_{1}+\xi_{2}x_{2}\right)}\eta\left(x\right)\d x\right|\leq C\norm{\eta}_{C^{3}\left(\R^{2}\right)}\left(\log\left(2+\left|\xi\right|\right)\right)^{\nu}\left(1+\left|\xi\right|\right)^{-1/h}.
\end{equation*}
\end{theorem}

Using Theorem \ref{theorem: linear perturbation}, we show
\begin{corollary}\label{coro: decomposition of torus}
    Let sets $\{\Sigma_{i}\}_{i=1}^{3}$ defined by \eqref{def: Sigmai} and fix $\left(x_{0},y_{0}\right)\in\left[0,2\pi\right]^{2}$. Then there exist a neighborhood of $\left(x_{0},y_{0}\right)$, denoted as $U_{\left(x_{0},y_{0}\right)}$, such that for all smooth test function $\chi$ supported in $U_{\left(x_{0},y_{0}\right)}$,
    \begin{equation}
        \left|\widetilde{I}\left(\xi,\eta;t;\chi\right)\right|\leq C_{\left(x_{0},y_{0}\right)}\norm{\mathbf{a}\chi}_{C^{3}\left(\R^{2}\right)}\left(1+\left|t\right|\right)^{-3/4}\;\;\mathrm{if}\;\left(x_{0},y_{0}\right)\in \Sigma_{3},
    \end{equation}
    \begin{equation}
        \left|\widetilde{I}\left(\xi,\eta;t;\chi\right)\right|\leq C_{\left(x_{0},y_{0}\right)}\norm{\mathbf{a}\chi}_{C^{3}\left(\R^{2}\right)}\left(1+\left|t\right|\right)^{-5/6}\;\;\mathrm{if}\;\left(x_{0},y_{0}\right)\in \Sigma_{2},
    \end{equation}
    \begin{equation}
        \left|\widetilde{I}\left(\xi,\eta;t;\chi\right)\right|\leq C_{\left(x_{0},y_{0}\right)}\norm{\mathbf{a}\chi}_{C^{3}\left(\R^{2}\right)}\left(1+\left|t\right|\right)^{-1}\;\;\mathrm{if}\;\left(x_{0},y_{0}\right)\in \Sigma_{1},
    \end{equation}
    for all $\left(\xi,\eta\right)\in\R^{2}$.
\end{corollary}
\begin{proof}[Proof of Corollary \ref{coro: decomposition of torus}]
    Without loss of generality, we just show the proof when $\left(x_{0},y_{0}\right)\in \Sigma_{3}$. Noting that there exists some $\left(\xi_{0},\eta_{0}\right)$ such that $\nabla\omega\left(x_{0},y_{0}\right)=\left(\xi_{0},\eta_{0}\right)$. By Theorem \ref{theorem: linear perturbation} there exist a neighborhood $U_{\left(x_{0},y_{0}\right)}$ of $\left(x_{0},y_{0}\right)$ and a constant $C_{\left(x_{0},y_{0}\right)}$, such that for any $\chi\in C^{\infty}_{c}\left(U_{\left(x_{0},y_{0}\right)}\right)$, and any $\left(t,t\left(\xi_{0}-\xi\right),t\left(\eta_{0}-\eta\right)\right)\in\R^{3}$, 
    \begin{equation*}
        \begin{aligned}
            \left|\widetilde{I}\left(\xi,\eta;t;\chi\right)\right|=&\left|\int_{U_{\left(x_{0},y_{0}\right)}}e^{-i\left(t\Omega\left(x,y;\xi_{0},\eta_{0}\right)+t\left(\xi_{0}-\xi\right)+t\left(\eta_{0}-\eta\right)\right)}\mathbf{a}\left(x,y\right)\chi\left(x,y\right)\d x\d y\right|\\
            \leq& C_{\left(x_{0},y_{0}\right)}\norm{\mathbf{a}\chi}_{C^{3}\left(\R^{2}\right)}\left(1+\left|t\right|\right)^{-3/4},
        \end{aligned}
    \end{equation*}
    as desired.
\end{proof}

To proceeding, we derive the following Proposition \ref{prop:fast decay}, which can be regarded as a easy variant of the stationary phase method, see, for example \cite{SM93, W03}.
\begin{proposition}\label{prop:fast decay}
     If $\mathrm{supp}\;\chi$ does not contain any critical points of the phase function $\Omega\left(x,y;\xi,\eta\right)$, then for any $M>0$,
    \begin{equation*}
        \left|\widetilde{I}\left(\xi,\eta;t;\chi\right)\right|\leq C\left(M,\chi,d\right)\JP{t}^{-M},
    \end{equation*}
    where $d$ is the infimum of $\left|t^{-1}\left(\xi,\eta\right)-\nabla_{\left(x,y\right)}\omega\right|$ over the support of $\chi$ with respect to $\left(x,y\right)$.
\end{proposition}

Applying Proposition \ref{prop: heights}, Corollary \ref{coro: decomposition of torus} and Proposition \ref{prop:fast decay}, we deduce the following Theorem \ref{theorem: POU+decay} in the spirit of \cite[Theorem 2.4]{BG17}.
\begin{theorem}\label{theorem: POU+decay}
    Let the sets $\{V_{i}\}_{i=0}^{3}$ be defined by \eqref{def: Vi}, $\chi$ be a smooth periodic function on $\left[0,2\pi\right]^{2}$, and the integral $I\left(\xi,\eta;t;\chi\right)$ defined by \eqref{def: General oscillatory integral on the torus}. Then for any fixed $\delta>0$, there exist constant $C_{0}$, $C_{1}$, $C_{2}$ and $C_{3}$ depending on $\chi$ such that
    \begin{itemize}
        \item For all $\left(\xi,\eta\right)$ with $\mathrm{dist}\left(\left(\xi,\eta\right),tV_{3}\right)\leq t\delta$, we have
            \begin{equation}\label{eq:A_{3} decay}
                \left|I\left(\xi,\eta;t;\chi\right)\right|\leq\frac{C_{3}\left(\chi\right)}{\left|t\right|^{3/4}}.
            \end{equation}
        \item For all $\left(\xi,\eta\right)$ with $\mathrm{dist}\left(x,tV_{3}\right)>t\delta$ and $\mathrm{dist}\left(\left(\xi,\eta\right),tV_{2}\right)\leq t\delta$, we have
            \begin{equation}
                \left|I\left(\xi,\eta;t;\chi\right)\right|\leq\frac{C_{2}\left(\chi,\delta\right)}{\left|t\right|^{5/6}}.
            \end{equation}
        \item For all $\left(\xi,\eta\right)$ with $\mathrm{dist}\left(\left(\xi,\eta\right),t\left(V_{2}\cup V_{3}\right)\right)> t\delta$ and $\mathrm{dist}\left(x,tV_{1}\right)\leq t\delta$, we have
            \begin{equation}
                \left|I\left(\xi,\eta;t;\chi\right)\right|\leq\frac{C_{1}\left(\chi,\delta\right)}{\left|t\right|}.
            \end{equation}
        \item For all $\left(\xi,\eta\right)$ with $\mathrm{dist}\left(\left(\xi,\eta\right),t\left(V_{1}\cup V_{2}\cup V_{3}\right)\right)>t\delta$, we have
            \begin{equation}
                \left|I\left(\xi,\eta;t;\chi\right)\right|\leq\frac{C_{0}\left(\chi,\delta,M\right)}{\left|t\right|^{M}}.
            \end{equation}
    \end{itemize}
\end{theorem}
\begin{proof}[Proof of Theorem \ref{theorem: POU+decay}]
    Let $\delta>0$ be a fixed number. When $\left(\xi,\eta\right)\in\R^{2}$ satisfies that $\mathrm{dist}\left(\left(\xi,\eta\right),t\left(V_{1}\cup V_{2}\cup V_{3}\right)\right)>t\delta$, i.e., $\mathrm{dist}\left(t^{-1}\left(\xi,\eta\right),V_{1}\cup V_{2}\cup V_{3}\right)>\delta$, then by taking use of Proposition \ref{prop:fast decay}, we obtain
    \begin{equation*}
        \left|\widetilde{I}\left(\xi,\eta;t;\chi\right)\right|\leq C\left(M,\chi,d\right)\JP{t}^{-M}
    \end{equation*}
    with $d=\delta$. Indeed, we note that $\nabla\omega$ is periodic on $\R^{2}$ and then
    \begin{equation*}
        \begin{aligned}
            &\inf_{\left(x,y\right)\in\mathrm{supp}\left(\mathbf{a}\cdot\chi\right)}\left|t^{-1}\left(\xi,\eta\right)-\nabla_{\left(x,y\right)}\omega\right|=\inf_{\left(x,y\right)\in\mathrm{supp}\left(\chi\right)}\left|t^{-1}\left(\xi,\eta\right)-\nabla_{\left(x,y\right)}\omega\right|>\delta.
        \end{aligned}
    \end{equation*}
    Below we always assume that $\mathrm{dist}\left(\left(\xi,\eta\right),t\left(V_{1}\cup V_{2}\cup V_{3}\right)\right)\leq t\delta$. 
    
    For any $\left(x,y\right)\in\left[0,2\pi\right]^{2}$, by fixing a neighborhood (in particular, we set this neighborhood be a ball) of $\left(x,y\right)$ which satisfies the inequalities in Corollary \ref{coro: decomposition of torus}, we can get a open cover $\{U_{\left(x,y\right)}\}_{\left(x,y\right)\in\left[0,2\pi\right]^{2}}$ of the compact set $\left[0,2\pi\right]^{2}$. Then one may extract a finite sub-cover $\{U_n\}_{1\le n\le N_1}$ and construct a subordinate partition of unity $\{\chi_n\}_{1\le n\le N_1}$ such that $\operatorname{supp}\chi_n\subset U_n$ for each $n$ and every $\chi_n$ satisfies the required periodicity condition. Since
    \begin{equation}
        \left|I\left(\xi,\eta;t;\chi\right)\right|\leq\sum_{1\leq n\leq N_{1}}\left|I\left(\xi,\eta;t;\chi\cdot\chi_{n}\right)\right|=\sum_{1\leq n\leq N_{1}}\left|\widetilde{I}\left(\xi,\eta;t;\chi\cdot\chi_{n}\right)\right|,
    \end{equation}
    and by Corollary \ref{coro: decomposition of torus}, each $\widetilde{I}\left(\xi,\eta;t;\chi\cdot\chi_{n}\right)$ satisfies that (here we use the lowest decay rate $-3/4$)
    \begin{equation*}
        \left|\widetilde{I}\left(\xi,\eta;t;\chi\cdot\chi_{n}\right)\right|\leq C_{n}\norm{\mathbf{a}\chi\cdot\chi_{n}}_{C^{3}\left(\R^{2}\right)}\left(1+\left|t\right|\right)^{-3/4},
    \end{equation*}
    we conclude that
    \begin{equation}\label{eq: -3/4 bound}
        \left|I\left(\xi,\eta;t;\chi\right)\right|\leq C\left(\chi\right)\JP{t}^{-3/4}.
    \end{equation}
    Inequality \eqref{eq: -3/4 bound} holds for any $\left(\xi,\eta\right)\in\R^{2}$, and we will use more detailed assumptions of the open cover to improve this inequality for some particular $\left(\xi,\eta\right)$. 
    
    We define the function $\mathcal{V}:\left[0,2\pi\right]^{2}\to \R^{2}$, where $\mathcal{V}\left(x,y\right)=\nabla_{\left(x,y\right)}\omega$ and hence it's uniformly continuous on $\left[0,2\pi\right]^{2}$. For given $\delta>0$, we choose some small enough $\epsilon>0$ depends on $\delta$ such that
    \begin{equation*}
        \mathrm{diam}\left(\mathcal{V}\left(B_{\epsilon}\right)\right)<\frac{\delta}{2}
    \end{equation*}
    for all ball $B_{\epsilon}$ of radius $\epsilon$. Applying Corollary \ref{coro: decomposition of torus} we get a refine open covering \begin{equation*}
        \left\{U_{\left(x,y\right)}\cap B_{\epsilon}\left(\left(x,y\right)\right)\right\}_{\left(x,y\right)\in\left[0,2\pi\right]^{2}}
    \end{equation*} 
    of $\left[0,2\pi\right]^{2}$, and after choosing a suitable finite sub-cover $\left\{U_{n}\right\}_{1\leq n\leq N_{2}}$, there is a partition of unity $\{\chi_{n}\}_{1\leq n\leq N_{2}}$ such that $\chi_{n}$ supports in $U_{n}$ for each $1\leq n\leq N_{2}$ and each $\chi_{n}$ satisfies the periodic condition. Noting that each subset $U_{n}$, $1\leq n\leq N_{2}$, is identified with a center $\left(x_{n},y_{n}\right)\in\left[0,2\pi\right]^{2}$, we give three index sets to represent those centers in $\Sigma_{i}$,
    \begin{equation*}
        J_{i}:=\left\{1\leq n\leq N_{2}:\left(x_{n},y_{n}\right)\in\Sigma_{i}\right\}\;\mathrm{where}\;i=1,2,3.
    \end{equation*}
    For every $\left(\xi,\eta\right)\in \R^{2}$ with $\mathrm{dist}\left(\left(\xi,\eta\right),tV_{3}\right)>t\delta$, we have
    \begin{equation*}
        \left|t^{-1}\left(\xi,\eta\right)-\nabla_{\left(x,y\right)}\omega\right|>\delta
    \end{equation*}
    for all $\left(x,y\right)\in \Sigma_{3}$ and this implies that $\left|t^{-1}\left(\xi,\eta\right)-\nabla_{\left(x,y\right)}\omega\right|>\delta/2$ for all $\left(x,y\right)\in\cup_{n\in J_{3}}U_{n}$. Proposition \ref{prop:fast decay} gives that for each $n\in J_{3}$, 
    \begin{equation*}
        \left|\widetilde{I}\left(\xi,\eta;t;\chi\cdot\chi_{n}\right)\right|\leq C\left(M,\chi,d\right)\JP{t}^{-M},
    \end{equation*}
    where we can set $d=\delta/2$. Therefore, 
    \begin{equation*}
        \begin{aligned}
            \left|I\left(\xi,\eta;t;\chi\right)\right|\leq&\sum_{1\leq n\leq N_{2}}\left|I\left(\xi,\eta;t;\chi\cdot\chi_{n}\right)\right|=\left(\sum_{n\in J_{3}}+\sum_{n\in J_{1}\cup J_{2}}\right)\left|\widetilde{I}\left(\xi,\eta;t;\chi\cdot\chi_{n}\right)\right|\\
            \leq&C\left(M,\chi,\delta\right)\JP{t}^{-M}+C\left(\chi\right)\JP{t}^{-5/6},
        \end{aligned}
    \end{equation*}
    for each $M\in\mathbb{N}$. In particular, for $M=1$, we have $\left|I\left(\xi,\eta;t;\chi\right)\right|\leq C\left(\chi,\delta\right)\JP{t}^{-5/6}$. 

    Finally, for those $\left(\xi,\eta\right)\in \R^{2}$ with $\mathrm{dist}\left(\left(\xi,\eta\right),t\left(V_{2}\cup V_{3}\right)\right)>t\delta$, we conclude
    \begin{equation*}
        \left|t^{-1}\left(\xi,\eta\right)-\nabla_{\left(x,y\right)}\omega\right|>\delta
    \end{equation*}
    for all $\left(x,y\right)\in \Sigma_{2}\cup \Sigma_{3}$ and this implies that $\left|t^{-1}\left(\xi,\eta\right)-\nabla_{\left(x,y\right)}\omega\right|>\delta/2$ for all $\left(x,y\right)\in\cup_{n\in J_{2}\cup J_{3}}U_{n}$. Using Proposition \ref{prop:fast decay} again, we derive that for each $n\in J_{2}\cup J_{3}$, 
    \begin{equation*}
        \left|\widetilde{I}\left(\xi,\eta;t;\chi\cdot\chi_{n}\right)\right|\leq C\left(M,\chi,d\right)\JP{t}^{-M},
    \end{equation*}
    where we can set $d=\delta/2$. Hence, inequality
    \begin{equation*}
        \begin{aligned}
            \left|I\left(\xi,\eta;t;\chi\right)\right|\leq&\sum_{1\leq n\leq N_{2}}\left|I\left(\xi,\eta;t;\chi\cdot\chi_{n}\right)\right|=\left(\sum_{n\in J_{2}\cup J_{3}}+\sum_{n\in J_{1}}\right)\left|\widetilde{I}\left(\xi,\eta;t;\chi\cdot\chi_{n}\right)\right|\\
            \leq&C\left(M^{\prime},\chi,\delta\right)\JP{t}^{-M^{\prime}}+C\left(\chi\right)\JP{t}^{-1}
        \end{aligned}
    \end{equation*}
    holds for each $M^{\prime}\in\mathbb{N}$. Setting $M^{\prime}=2$, we have $\left|I\left(\xi,\eta;t;\chi\right)\right|\leq C\left(\chi,\delta\right)\JP{t}^{-1}$. We complete the proof.
\end{proof}
As a direct consequence, if one only concern about the worst possible decay among all $x\in\Z^{2}$, the simpler statement is 
\begin{corollary}\label{coro: sharp decay-R^2}
   Taking $\chi=\mathbbm{1}_{\left[0,2\pi\right]^{2}}$ in \eqref{def: General oscillatory integral on the torus}, one has
   \begin{equation*}
       \norm{I\left(\xi,\eta;t;\mathbbm{1}_{\left[0,2\pi\right]^{2}}\right)}_{L^{\infty}_{\left(\xi,\eta\right)}\left(\R^{2}\right)}\leq C\JP{t}^{-3/4}.
   \end{equation*}
   for some constants $C>0$.
\end{corollary}

Corollary \ref{coro: sharp decay-R^2} indicates that our main estimate is sharp in the following sense: there are vectors $\left(\xi_{0},\eta_{0}\right)\in \R^d$ and non-zero constant $c$ such that 
\begin{equation}\label{sharp}
    \lim_{t\to\infty}t^{3/4}\left|I\left(t\xi_{0},t\eta_{0};t;\mathbbm{1}_{\left[0,2\pi\right]^{2}}\right)\right|=c,
\end{equation} since the following estimate proved by Ikromov and M\"uller holds.
\begin{theorem}\cite[Theorem 1.3]{IM11}\label{theorem: sharpness}
    Let us put 
    \begin{equation*}
        J_{\pm}\left(\lambda\right):=\int_{\R^{2}}e^{\pm i\lambda\varphi\left(x_{1},x_{2}\right)}\eta\left(x\right)\d x,\;\;\lambda>0,
    \end{equation*}
    with $\varphi$ and $\eta$ as in Theorem \ref{theorem: linear perturbation}. If the principle face $\pi\left(\varphi^{a}\right)$ of $\varphi$, when given in adapted coordinates, is a compact set (i.e., a compact edge or a vertex), then there exists a neighborhood $\Omega$ of the origin such that for every $\eta$ supported in $\Omega$ the following limits
    \begin{equation*}
        \lim_{\lambda\to\pm\infty}\frac{\lambda^{1/h}}{\left(\log\lambda\right)^{\nu}}J_{\pm}\left(\lambda\right)=c_{\pm}\eta\left(0\right)
    \end{equation*}
    exist, where the constants $c_{\pm}$ are non-zero and depend on the phase function $\varphi$ only.
\end{theorem}

 \section{Proofs of Propositions \ref{prop: velocity space} and \ref{prop: heights}}\label{sec3}

For simplicity, the following notion will be used.
\begin{definition}
	Let weight $\kappa=\left(\kappa_{1},\kappa_{2},\dots,\kappa_{d}\right)$ with $\kappa_{i}>0$ for $i=1,2,\dots,d$. We say a function $f:\R^{d}\to\R$ is $\kappa-$homogeneous of degree $r$ if
    \begin{equation*}
		f\left(\lambda^{\kappa_{1}}x_{1},\lambda^{\kappa_{2}}x_{2},\dots,\lambda^{\kappa_{d}}x_{d}\right)=\lambda^{r}f\left(x_{1},x_{2},\dots,x_{d}\right).
	\end{equation*}
\end{definition}
For an analytic function $f$ defined on $\R^{2}$ near the origin, we can choose a suitable weight $\kappa=\left(\kappa_{1},\kappa_{2}\right)$ with $0<\kappa_{1},\kappa_{2}\leq1$ such that its Taylor series can be written as
\begin{equation*}
    f\left(x_{1},x_{2}\right)=f_{\mathrm{pr}}\left(x_{1},x_{2}\right)+\mathrm{sum\;of\;the\;monimials\;with\;\kappa-homogeneous\;of\;degree>1},
\end{equation*}
where the principle part $f_{\mathrm{pr}}$ is $\kappa-$homogeneous of degree 1. This homogeneous polynomial helps us to analyze the Newton polyhedron of $f$ in $\left(x_{1},x_{2}\right)$ local coordinate system.

Now we begin to proof Propositions \ref{prop: velocity space} and \ref{prop: heights}. 
 
	\begin{proof}[Proof of Propositions \ref{prop: velocity space} and \ref{prop: heights}]
		
		Recall the phase function \begin{equation}
			\begin{aligned}
				\Omega\left(x,y;\xi,\eta\right)=&8-2\cos x-2\cos y-2\cos\left(x+y\right)-2\cos\left(x-y\right)-\left(x\xi+y\eta\right).
			\end{aligned}  
		\end{equation}
For a fixed $\left(x_{0},y_{0}\right)\in\left[0,2\pi\right]^{2}$, there exists  $\left(\xi_{0},\eta_{0}\right)\in\R^{2}$ with $\left(\xi_{0},\eta_{0}\right)=\nabla_{\left(x_{0},y_{0}\right)}\omega$,  such that $\left(x_{0},y_{0}\right)$ is a critical point of $\Omega\left(x,y;\xi_{0},\eta_{0}\right)$. That is, we derive that \begin{equation*}
			\begin{cases}
				2\sin x_{0}+2\sin\left(x_{0}+y_{0}\right)+2\sin\left(x_{0}-y_{0}\right)=\xi_{0},\\
				2\sin y_{0}+2\sin\left(x_{0}+y_{0}\right)-2\sin\left(x_{0}-y_{0}\right)=\eta_{0}.\\
			\end{cases}
		\end{equation*}
		 The type of singularity of $\Omega\left(x,y;\xi_{0},\eta_{0}\right)$ at $\left(x_{0},y_{0}\right)$ is identical to that of 
		\begin{equation}
			\Omega^{\prime}\left(x,y;\xi_{0},\eta_{0}\right) := \Omega\left(x+x_{0},y+y_{0};\xi_{0},\eta_{0}\right)-\Omega\left(x_{0},y_{0};\xi_{0},\eta_{0}\right).
		\end{equation}
at origin. For notational convenience, we denote $\Omega^{\prime}\left(x,y;\xi_{0},\eta_{0}\right)$ by $\Omega(x,y)$ and write that
		\begin{equation}\label{sici}
			\begin{cases}
				s_{1}:=\sin x_{0},\;c_{1}:=\cos x_{0},\\
				s_{2}:=\sin y_{0},\;c_{2}:=\cos y_{0}.\\
			\end{cases}
		\end{equation}
		Trigonometric identities yeild
		\begin{equation*}
			\begin{aligned}
				\Omega\left(x,y\right)
				=&2c_{1}\left(1-\cos x\right)+2s_{1}\left(\sin x-x\right)+2c_{2}\left(1-\cos y\right)+2s_{2}\left(\sin y-y\right)
				\\
				&+2\left(c_{1}c_{2}-s_{1}s_{2}\right)\left(1-\cos\left(x+y\right)\right)+2\left(c_{1}c_{2}+s_{1}s_{2}\right)\left(1-\cos\left(x-y\right)\right)\\
				&+2\left(c_{1}s_{2}+c_{2}s_{1}\right)\left(\sin\left(x+y\right)-\left(x+y\right)\right)+2\left(c_{1}s_{2}-c_{2}s_{1}\right)\left(\sin\left(x-y\right)-\left(x-y\right)\right).
			\end{aligned}
		\end{equation*}
		Its Hessian matrix  
		\begin{equation*}
			\begin{aligned}
				&\mathrm{Hess}_{\left(0,0\right)}\,\Omega
				=2 
				\begin{pmatrix}
					c_1\left(1+2c_{2}\right) &-2s_{1}s_{2} \\
					-2s_{1}s_{2}&c_{2}\left(1+2c_{1}\right)\\
				\end{pmatrix},
			\end{aligned}
		\end{equation*}
		and the corresponding determinant 
		\begin{align}\label{eq:det}
			\mathrm{det}(\mathrm{Hess}_{(0,0)}\Omega)=4\big(c_1c_2(1+2c_1)(1+2c_2)-4s_1^2s_2^2\big).    
		\end{align}
If the Hessian matrix is non-degenerate, then any local coordinate system is adapted, and hence $h=1$. In the following discussions, it is enough to consider with the degenerate critical points and set 
\begin{equation}\label{eq:det0}
c_1c_2(1+2c_1)(1+2c_2)-4s_1^2s_2^2=0.
\end{equation}
The available $c_1$ and $c_2$ satisfy \eqref{eq:det0} see Figure \ref{fig:det}.

\begin{figure}[H] 
  \centering 
  \includegraphics[width=0.6\textwidth]{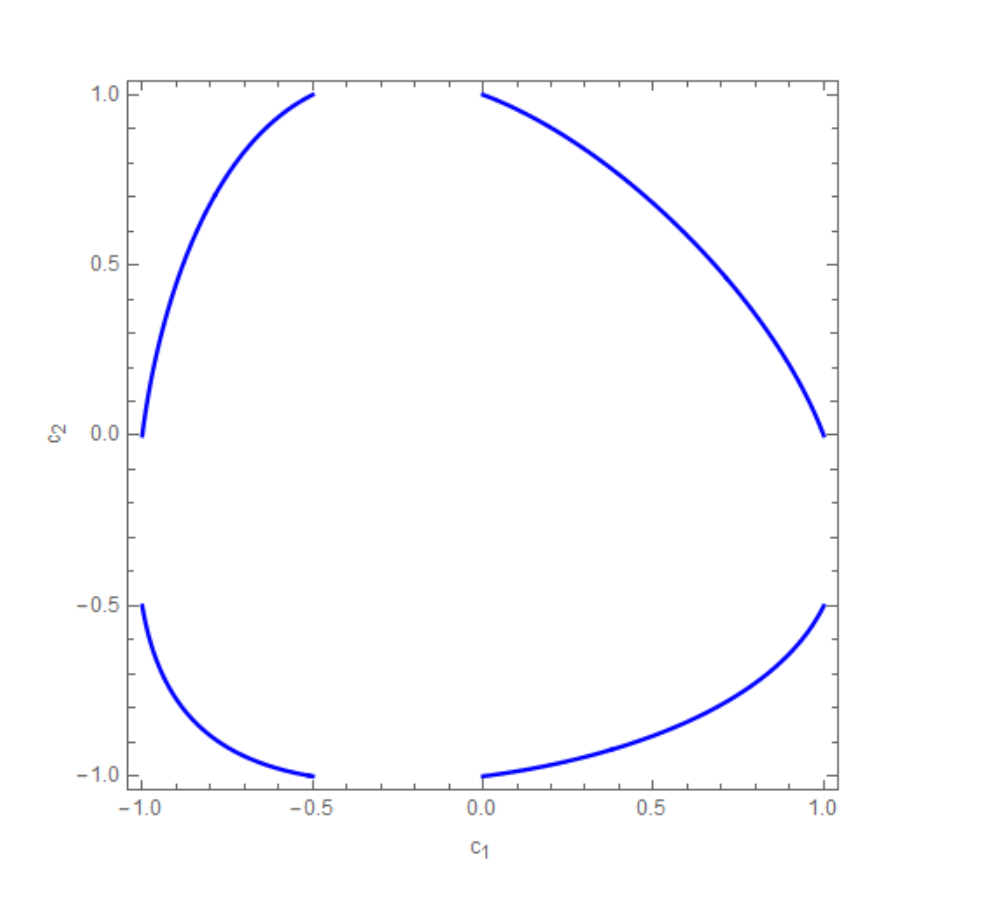} 
  \caption{$c_1c_2(1+2c_1)(1+2c_2)-4s_1^2s_2^2=0$} 
  \label{fig:det} 
\end{figure}

 We claim that $\mathrm{Hess}_{\left(0,0\right)}\Omega$ is not vanishing. Indeed, if $\mathrm{Hess}_{\left(0,0\right)}\Omega=0$, with loss of generality, for off-diagonal elements, we can assume $s_1=0$, then $c_1=\pm 1$ and $c_2=-1/2$. This is contradict to $c_{2}\left(1+2c_{1}\right)=0$, as claimed.
  
		Noting that $c_1=c_2=0$ is not a solution of equation \eqref{eq:det0}, we divide the proof into the follow cases.
		
		{\bf Case I.} Suppose exactly one $c_{i}$ is nonzero. With loss of generality, we assume $c_{2}=0$, then the Hessian matrix is
		\begin{equation*}
			2 
			\begin{pmatrix}
				c_{1} & -2s_{1}s_{2}\\
				-2s_{1}s_{2}&0\\
			\end{pmatrix}.
		\end{equation*}
Equation \eqref{eq:det0} implies $s_1^2=0$, i.e.,
$c_{1}=\pm1$. The Taylor expansion of $\Omega\left(x,y\right)$ is	\begin{equation*}
			c_{1}x^{2}-\frac{c_{1}}{12}x^{4}-\frac{s_{2}}{3}y^{3}+\frac{s_{2}}{60}y^{5}-\frac{c_{1}s_{2}}{3}\left(x+y\right)^{3}-\frac{c_{1}s_{2}}{3}\left(x-y\right)^{3}+\cdots.
	\end{equation*}
		By taking weight $\left(1/2,1/3\right)$ we know the principal part is $c_{1}x^{2}-\frac{s_{2}}{3}y^{3}$, which means that origin is an $A_{2}$ type singularity. It follows from Proposition \ref{proposition: A_{3}-type singularity} that the coordinate system is adapted and the principle face lies in the line $3n_1+2n_2=6$. Therefore the height is $6/5$. 
		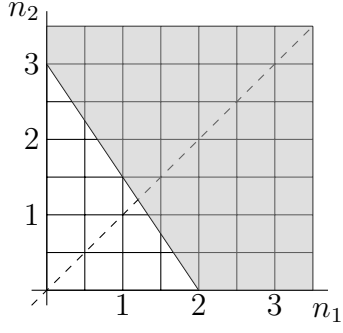
\begin{figure}[H]
			\centering
			\begin{tikzpicture}
				\draw (-0.2,0)--(3.7,0);
				\draw (0,-0.2)--(0,3.7);
				\draw (0,3)--(2,0);
				\draw [dashed](-0.2,-0.2)--(3.5,3.5);
				\draw [step=0.5cm] (0,0) grid (3.5,3.5);
				\draw (1,-0.2) node{1};
				\draw (2,-0.2) node{2};
				\draw (3,-0.2) node{3};
				\draw (-0.2,1) node{1};
				\draw (-0.2,2) node{2};
				\draw (-0.2,3) node{3};
				\fill [fill=gray!50][opacity=0.5] (0,3.5)--(0,3)--(2,0)--(3.5,0)--(3.5,3.5)-- cycle;
			\draw (3.7,-0.3) node{$n_{1}$};
                \draw (-0.3,3.7) node{$n_{2}$};
   \end{tikzpicture}
			\caption{Newton Polyhedron in Case I and Case IIa}\label{figure: NP1}
		\end{figure}

		{\bf Case II.} If $c_{1},c_{2}\neq 0$. 
	
		{\bf Case IIa.} Suppose $c_{1}=-1/2$. Taking use of \eqref{eq:det0}, we obtain $c_{2}= \pm 1$ and $s_{2}= 0$. The Taylor expansion is 
		$$c_1(1+2c_2)x^2-\frac{s_1}{3}x^3-2c_2x^2y-\frac{2s_1c_2}{3}y^3+\cdots.$$
		The fact that origin is an $A_{2}$ type singularity comes from setting weight $\left(1/2,1/3\right)$ with principal part $c_1(1+2c_2)x^{2}-\frac{2s_1c_2}{3}y^3$. By Proposition \ref{proposition: A_{3}-type singularity}, we deduce that the coordinate system is adapted and the principle face lies in the line $3n_1+2n_2=6$, which implies that the height is $6/5$. The discussion in the case of $c_{2}=-1/2$ is similar. 
		
		{\bf Case IIb.} Suppose $c_1\neq-1/2$ and $c_2\neq-1/2$. The Taylor expansion is 
  $$c_1(1 + 2c_2) x^2 + c_2(1 + 2c_1 ) y^2 - 4s_1 s_2 x y - \frac{1}{3}(s_1 + 2c_1 s_2) x^3 - 2c_2 s_1 x^2 y - 2c_1 s_2 x y^2 - \frac{1}{3}(2c_2 s_1 + s_2) y^3+\cdots.$$
   Changing of variables by $u=x-\frac{2s_1s_2}{c_1(1+2c_2)}y$ and $v=y$, we have
  \begin{align*}
      \Omega\left(u,v\right)=&\left(2c_1c_2+c_1\right)u^2+\alpha uv^2
      +\left(-\frac{2s_1^2s_2}{c_1(2c_2+1)}-\frac{4s_1s_2^2}{2c_2+1}-2c_2s_1 \right)u^2v\\
      &+\left(-\frac{2c_1s_2}{3}- \frac{s_1}{3}\right)u^3+\beta v^3+\gamma v^4 +\cdots.
  \end{align*}
  where \begin{equation}\label{coe:uv2}
      \alpha=-\frac{4s_1^3s_2^2}{c_1^2(2c_2+1)^2}-\frac{8s_1^2s_2^3}{c_1(2c_2+1)^2}-\frac{8c_2s_1^2s_2}{c_1(2c_2+1)}-2c_1s_2,
  \end{equation}
\begin{equation}\label{coe:v3}
   \beta=-\frac{8s_1^4s_2^3}{3c_1^3(2c_2+1)^3}-\frac{16s_1^3s_2^4}{3c_1^2(2c_2+1)^3}-\frac{8c_2s_1^3s_2^2}{c_1^2(2c_2+1)^2}-\frac{4s_1s_2^2}{2c_2+1} -\frac{2c_2s_1}{3}-\frac{s_2}{3},
\end{equation}
and \begin{equation}\label{coe:v4}
   \gamma= \frac{16 s_1^4 s_2^4}{3 c_1^3 (2 c_2+1)^3}-\frac{8 c_2 s_1^4 s_2^4}{3 c_1^3 (2 c_2+1)^4}-\frac{4 s_1^4 s_2^4}{3 c_1^3 (2 c_2+1)^4}+\frac{4 s_1^2 s_2^2}{3 c_1 (2 c_2+1)}-\frac{4 c_2 s_1^2 s_2^2}{c_1 (2 c_2+1)^2}-\frac{c_1 c_2}{6} - \frac{c_2}{12}.
\end{equation}
We  denote that 
  \begin{equation*}
      \Gamma_1:=\left\{\left(c_1,c_2\right)\in\left(\left[-1,1\right]\setminus\left\{0,-1/2\right\}\right)^{2}:\mathrm{Equation}\;\eqref{eq:det0}\;\mathrm{holds}\right\},
  \end{equation*}
  \begin{equation*}
      \Gamma_2:=\left\{\left(c_1,c_2\right)\in\left(\left[-1,1\right]\setminus\left\{0,-1/2\right\}\right)^{2}:\beta=0\right\},
  \end{equation*}
    \begin{equation*}
      \Gamma_3:=\left\{\left(c_1,c_2\right)\in\left(\left[-1,1\right]\setminus\left\{0,-1/2\right\}\right)^{2}:\alpha=0\right\},
  \end{equation*}
  and
      \begin{equation*}
      \Gamma_4:=\left\{\left(c_1,c_2\right)\in\left(\left[-1,1\right]\setminus\left\{0,-1/2\right\}\right)^{2}:\alpha^2-4(2c_1c_2+c_1)\gamma=0\right\}.
  \end{equation*}
  
Due to the calculations provided in the Appendix \ref{appendix}, we derive
\begin{itemize}\label{it:1}
    \item $\Gamma_{1}\cap\Gamma_{2}\neq \emptyset$,
    \item $\Gamma_{1}\cap\Gamma_{2}\cap\Gamma_{3}=\emptyset$,
    \item $\Gamma_{1}\cap\Gamma_{2}\cap\Gamma_{4}=\emptyset$.
\end{itemize}
There are only two possible types of singularity arise. When the coefficient of $v^3$ does not vanish, in this case, after setting weight $\left(1/2,1/3\right)$, we deduce this is an $A_{2}$ type singularity. The Newton polyhedron is drown in Figure \ref{figure: NP1}. 

In the case of the coefficient of $uv^2$ is non-zero, and by taking weight $\left(1/2,1/4\right)$, we see origin is an $A_{3}$ type singularity and part of the Newton polyhedron is convex hull of $\left(\R^{2}_{+}+\left(1,2\right)\right)\cup\left(\R^{2}_{+}+\left(2,0\right)\right)$, drawn in the following Figure \ref{figure: NP2bA4}.
\begin{figure}[H]
	\centering
	\begin{tikzpicture}
		\draw (-0.2,0)--(3.7,0);
		\draw (0,-0.2)--(0,3.7);
		\draw (1,2)--(2,0);
            \draw [dashed](0,4)--(1,2);
            \draw [dashed](-0.2,-0.2)--(3.5,3.5);
		\draw [step=0.5cm] (0,0) grid (3.5,4.6);
		\draw (1,-0.2) node{1};
		\draw (2,-0.2) node{2};
		\draw (3,-0.2) node{3};
		\draw (-0.2,1) node{1};
		\draw (-0.2,2) node{2};
		\draw (-0.2,3) node{3};
            \draw (-0.2,4) node{4};
		\fill [fill=gray!50][opacity=0.5] (1,4.5)--(1,2)--(2,0)--(3.5,0)--(3.5,4.5)-- cycle;
		\draw (3.7,-0.3) node{$n_{1}$};
                \draw (-0.3,4.7) node{$n_{2}$};\end{tikzpicture}
		\caption{Newton Polyhedron in adapted coordinates}\label{figure: NP2bA4}
  \end{figure}
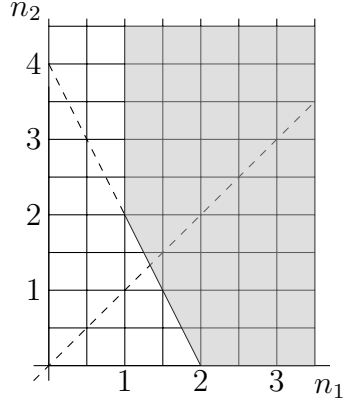
\noindent Noting that the principal part is
\begin{equation*}
    \begin{aligned}
        \Omega_{\mathrm{pr}}\left(u,v\right)=\left(2c_1c_2+c_1\right)u^2+\alpha uv^2+\gamma v^4,
    \end{aligned}
\end{equation*} we obtain the principal face lied on the line $2n_1+n_2=4$. In order to apply Proposition \ref{prop:V0.8}, we examine the associated polynomial
\begin{equation*}
    \begin{aligned}
        P_{\mathrm{pr}}\left(u,v\right)=\left(2c_1c_2+c_1\right)u^2+\alpha u+\gamma.
    \end{aligned}
\end{equation*}
As $\Gamma_{1}\cap\Gamma_{2}\cap\Gamma_{4}=\emptyset$, we have that the discriminant of $P_{\mathrm{pr}}\left(u,v\right)$, 
\begin{equation}\label{coe:discriminant}
D(P_{\mathrm{pr}}(u,v))=\alpha^2-4(2c_1c_2+c_1)\gamma,
\end{equation}
is non-zero. That is, $P_{\mathrm{pr}}(u,v)$ does not have a real root of multiplicity larger than $1$. By Proposition \ref{prop:V0.8}, we obtain that this is an adapted coordinate and the height $h$ is $4/3$.
\end{proof}

	\section{Strichartz estimate and some applications}\label{sec4}
	\subsection{Strichartz estimate}\label{subsec3.1}
	We first fix some standard notations.
	\begin{definition}
		For the function $f:\Z^{d}\times\R_{>0}\to\C$, denote by $\norm{f}_{L^{q}_{t}l^{p}_{x}\left(\R_{>0}\times\Z^{d}\right)}$ be the space-time mix Lebesgue norm of $f$, defined as 
		\begin{equation*}
			\norm{f}_{L^{q}_{t}l^{p}_{n}\left(\left(0,T\right)\times\Z^{d}\right)}:=\left(\int_{\left(0,T\right)}\left(\sum_{n\in\Z^{d}}\left|f\left(n,t\right)\right|^{p}\right)^{q/p}\d t\right)^{1/q}.
		\end{equation*}
	\end{definition}

 Let $u\in C^1([0,+\infty);l^2(\mathbb{Z}^3))$ be a solution of the equation \ref{eq:DS},  the following energy identity comes from the well-known $l^2$ conservation law:
 \begin{equation*}
		\norm{e^{it\Delta_{disc}}u_{0}}_{l^{2}\left(\Z^{3}\right)}=\norm{u_{0}}_{l^{2}\left(\Z^{3}\right)},
	\end{equation*}
	where $t\in\R$ and $u_{0}\in l^{2}\left(\Z^{3}\right)$. To give its proof, one just needs to check that
	\begin{equation*}
		\frac{\d}{\d t}\left(\norm{e^{it\Delta_{\mathrm{disc}}}u_{0}}_{l^{2}\left(\Z^{3}\right)}^{2}\right)=2\RE\left(\Inpro{\partial_{t}u\left(t\right)}{u\left(t\right)}{l^{2}\left(\Z^{3}\right)}\right)=2\RE\left(\Inpro{i\Delta_{\mathrm{disc}}u\left(t\right)}{u\left(t\right)}{l^{2}\left(\Z^{3}\right)}\right)=0.
	\end{equation*}
	Moreover, it's easy to verify that $\left(e^{it\Delta_{\mathrm{disc}}}\right)^{*}=e^{-it\Delta_{\mathrm{disc}}}$. Indeed, for any $u,v\in l^{2}\left(\Z^{3}\right)$, we have
	\begin{equation*}
		\Inpro{e^{it\Delta_{\mathrm{disc}}}u}{v}{l^{2}\left(\Z^{3}\right)}=\Inpro{\widehat{u}\psi_{t}}{\widehat{v}}{l^{2}\left(\Z^{3}\right)}=\Inpro{\widehat{u}}{\psi_{-t}\widehat{v}}{l^{2}\left(\Z^{3}\right)}=\Inpro{u}{e^{-it\Delta_{\mathrm{disc}}}v}{l^{2}\left(\Z^{3}\right)}.
	\end{equation*}
	Using this fact and the $l^{1}\to l^{\infty}$ estimate we obtain that
	\begin{equation*}
		\norm{e^{it\Delta}\left(e^{is\Delta}\right)^{*}u}_{l^{1}\left(\Z^{3}\right)}=\norm{e^{i\left(t-s\right)\Delta}u}_{l^{1}\left(\Z^{3}\right)}\lesssim\JP{t-s}^{-1}\norm{u}_{l^{\infty}\left(\Z^{3}\right)}.
	\end{equation*}
	For deriving the Strichartz estimate of \eqref{equation: inhomogegeneous Schrodinger equation}, we need to apply the following well-known result from Keel and Tao which has shown in \cite{KT98}. We recall the definition of $\sigma-$admissible as follows.
	\begin{definition}
		Let $\sigma$ be a given parameter. We say that a pair of exponents $\left(q,r\right)$ is Strichartz $\sigma-$admissible, if $q,r\ge 2$, $\left(q,r,\sigma\right)\neq\left(2,\infty,1\right)$ and $1/q+\sigma/r\leq \sigma/2$.
	\end{definition}
	\begin{theorem}[Keel and Tao, Theorem 1.2. \cite{KT98} ]\label{thm:KT}
		Let $H$ be a Hilbert space, $\left(X,\d x\right)$ be a measure space and $U\left(t\right):H\to L^{2}\left(X\right)$ be a one parameter family of mappings, which obey the energy estimate
		\begin{equation*}
			\norm{U\left(t\right)f}_{L^{2}_{x}}\lesssim \norm{f}_{H}
		\end{equation*}q
		and the decay estimate
		\begin{equation*}
			\norm{U\left(t\right)\left(U\left(s\right)\right)^{*}g}_{L^{\infty}\left(X\right)}\lesssim\JP{t-s}^{-\sigma}\norm{g}_{L^{1}\left(X\right)},
		\end{equation*}
		for some $\sigma>0$. Then,
		\begin{equation*}
			\begin{aligned}
				&\norm{U\left(t\right)f}_{L_{t}^{q}L_{x}^{r}}\lesssim\norm{f}_{L^{2}},\\
				&\norm{\int\left(U\left(t\right)\right)^{*}F\left(t,\cdot\right)\d t}_{H}\lesssim \norm{F}_{L^{q^{\prime}}_{t}L^{r^{\prime}}_{x}},\\
				&\norm{\int^{t}_{0}U\left(t\right)\left(U\left(s\right)\right)^{*}F\left(s,\cdot\right)\d s}_{L^{q}_{t}L^{r}_{x}}\lesssim\norm{F}_{L^{\tilde{q}^{\prime}}_{t}L^{\tilde{r}^{\prime}}_{x}},\\
			\end{aligned}
		\end{equation*}
		where exponent pairs $\left(q,r,\sigma\right)$ and $\left(\tilde{q},\tilde{r},\sigma\right)$ are Strichartz $\sigma-$admissible.
	\end{theorem}
	Using the Keel-Tao argument and our $l^{1}\to l^{\infty}$ estimate, we get that
	
	\begin{proof}[Proof of Theorem \ref{equation: inhomogegeneous Schrodinger equation}]
		By the Duhamel's principle we have that for the solution $u$, 
		\begin{equation*}
			u\left(t\right)=e^{it\Delta_{\mathrm{dist}}}u_{0}+\int_{0}^{t}e^{i\left(t-s\right)\Delta_{\mathrm{dist}}}F\left(s\right)\d s.
		\end{equation*}
		Setting $X=\Z^{3}$, measure $\d x$ be the counting measure and $H$ be the Hilbert space $l^{2}\left(\Z^{3}\right)$, it follows from Theorem \ref{thm:KT} that\begin{equation*}
			\begin{aligned}
				\norm{u}_{L^{q}_{t}l^{r}_{n}\left(\left(0,T\right)\times\Z^{3}\right)}\leq&\norm{e^{it\Delta_{\mathrm{dist}}}u_{0}}_{L^{q}_{t}l^{r}_{n}\left(\left(0,T\right)\times\Z^{3}\right)}+\norm{\int_{0}^{t}e^{i\left(t-s\right)\Delta_{\mathrm{dist}}}F\left(s\right)\d s}_{L^{q}_{t}l^{r}_{n}\left(\left(0,T\right)\times\Z^{3}\right)}\\
				\lesssim&\norm{u_{0}}_{l^{2}\left(\Z^{3}\right)}+\norm{F}_{L^{\tilde{q}^{\prime}}_{t}l^{\tilde{r}^{\prime}}_{n}\left(\left(0,T\right)\times\Z^{3}\right)},
			\end{aligned}
		\end{equation*}
		as desired.
	\end{proof}

	\subsection{global well-posedness for small data}\label{subsec3.2}
	
	With the help of Theorem \ref{thm:Strichartz}, we can establish the global existence of the solution to discrete nonlinear Schr\"odinger equations (DNLS) with a small initial data on LKG.

     \begin{eqnarray}\label{DNLS}
   \left\{\begin{aligned}&\partial_tu(x,t)-i\Delta_{\mathrm{disc}}u(x,t)\pm i|u(x,t)|^{2a}u(x,t)=0,~~\forall (x,t)\in\Z^{3}\times[0,\infty) \\
    &u(\cdot,0)=u_0\in l^2(\mathbb{Z}^3).
 \end{aligned}
   \right.
   \end{eqnarray}
\begin{theorem}
    Let $a\ge 12/13$. If there exists $\epsilon>0$ and a positive constant $C$ such that $\|u_0\|_{l^2}\le \epsilon$, then \eqref{DNLS} has a global solution $u\in C^1((0,\infty];l^2(\mathbb{Z}^3))$. Moreover,
    \[
    \|u_t\|_{L^ql^r\left(\left(0,\infty\right)\times\Z^{3}\right) }\le C\epsilon,
    \]
holds for all $13/12$-admissible pairs $(q,r)$. In particular, 
\[
\lim_{t\rightarrow\infty}\|u(t)\|_{l^r(\mathbb{Z}^3)}=0, \ \mathrm{for \ all}\  r>2.
\]
\end{theorem}
\begin{proof}
   We define a metric space
    \[
    \mathcal{X}=\{u(t):\sup_{(q,r):\mathrm{\frac{13}{12}-admissible}}\|u\|_{L^ql^r\left(\left(0,\infty\right)\times\Z^{3}\right)}<2C\|u_0\|_{l^2(\mathbb{Z}^3)}\},
    \]
 where $C$ is the constant in Theorem \ref{thm:Strichartz}. In the following, we will find a fixed point in $\mathcal{X}$ of an operator $\mathbf{A}$ defined by 
    \[
\mathbf{A}(u)=e^{it\Delta_{\mathrm{disc}}}u_0 \pm i\int_0^te^{i(t-s)\Delta_{\mathrm{disc}}}|u|^{2a}u(s)\mathrm{d}s.
    \]
    Setting $(\tilde{q},\tilde{r})=(\infty,2)$, and by Theorem \ref{thm:Strichartz} we have
    \begin{align*}
        \sup_{(q,r):\mathrm{\frac{13}{12}-admissible}}\|\mathbf{A}u(t)\|_{L^ql^r\left(\left(0,\infty\right)\times\Z^{3}\right)}&\le C\|u_0\|_{l^2(\mathbb{Z}^3)}+C_1\|u^{2a+1}\|_{L^1l^2\left(\left(0,\infty\right)\times\Z^{3}\right)} \\
         &\le C\|u_0\|_{l^2(\mathbb{Z}^3)}+C_1\|u\|^{2a+1}_{L^{2a+1}l^{2(2a+1)}\left(\left(0,\infty\right)\times\Z^{3}\right)}.    \end{align*}
Since $a\ge\frac{12}{13}$, we have $(2a+1,2(2a+1))$ is $13/12$-admissible. Then we obtain
\begin{align*}
    \sup_{(q,r):\mathrm{\frac{13}{12}-admissible}} \|\mathbf{A}u(t)\|_{L^ql^r\left(\left(0,\infty\right)\times\Z^{3}\right)}\le C\|u_0\|_{l^2(\mathbb{Z}^3)}+C_1\|u\|_{\mathcal{X}}^{2a+1}
    \le C\|u_0\|_{l^2(\mathbb{Z}^3)}+C_1(2C\epsilon)^{2a+1}.
\end{align*}
    Letting $\epsilon>0$ be sufficient small,  we derive 
    \[
    \sup_{(q,r):\mathrm{\frac{13}{12}-admissible}}\|\mathbf{A}u(t)\|_{L^ql^r\left(\left(0,\infty\right)\times\Z^{3}\right)}\le2C\|u_0\|_{l^2(\mathbb{Z}^3)},
    \]
which implies $\mathbf{A}(\mathcal{X})\subseteq \mathcal{X}$ given that $\|u_0\|_{l^2(\mathbb{Z}^3)}$ is small enough. Similarly, 
letting $u,\tilde{u}\in \mathcal{X}$ with the same initial data $u(\cdot,0)=\tilde{u}(\cdot,0)=u_0,$ we conclude
\begin{align*}
    &\sup_{(q,r):\mathrm{\frac{13}{12}-admissible}}\|\mathbf{A}u-\mathbf{A}\tilde{u}\| \\
    &\le C_0\|u-\tilde{u}\|_{L^{2a+1}l^{2(2a+1)}\left(\left(0,\infty\right)\times\Z^{3}\right)}(\|u\|_{L^{2a+1}l^{2(2a+1)}\left(\left(0,\infty\right)\times\Z^{3}\right)}+\|\tilde{u}\|_{L^{2a+1}l^{2(2a+1)}\left(\left(0,\infty\right)\times\Z^{3}\right)})\\
&\le2CC_1\|u-\tilde{u}\|_{L^{2a+1}l^{2(2a+1)}\left(\left(0,\infty\right)\times\Z^{3}\right)}\|u_0\|_{l^2(\mathbb{Z}^3)}^{2a}.
\end{align*}
Therefore $\mathbf{A}$ is a contraction map given that the $l^2$ norm of $u_0$ is sufficiently small.
This proves the existence of a global solution \eqref{DNLS}.
\end{proof}

	\appendix
\section{Mathematica Scripts for System Consistency Analysis}\label{appendix}
\label{app:mathematica_consistency}

This appendix presents the Mathematica script \cite{Mathematica} used to test the consistency of the algebraic equations derived in Section \ref{sec2} Case (IIb). The analysis involves constructing three distinct systems of equations and using the \texttt{Reduce} function to determine their solvability in the real domain.

The function \texttt{Reduce[expr, vars, Reals]} searches for all real-valued solutions for the variables \texttt{vars} that satisfy the expressions \texttt{expr}. Its methodology, based on Cylindrical Algebraic Decomposition (CAD), reliably finds the complete solution set for polynomial systems. When solutions exist, \texttt{Reduce} returns a logical expression defining the solution space. Critically, if no solution exists, it returns \texttt{False}, which constitutes a definitive proof of the system's inconsistency. This rigorous approach ensures the reliability of our conclusions.

The variables \(c_1, s_1, c_2, s_2\) are defined in \eqref{sici}. The core of the analysis involves three systems:
\begin{itemize}
    \item \textbf{System 1}: Equation \eqref{eq:det0} holds and the coefficient of $v^3$, i.e. \eqref{coe:v3} vanishes.
    \item \textbf{System 2}: Equation \eqref{eq:det0} holds, and the coefficients of $v^3$ and $uv^2$, i.e. \eqref{coe:v3} and \eqref{coe:uv2} vanish.
    \item \textbf{System 3}: Equation \eqref{eq:det0} holds, the coefficient of $v^3$ vanishes, i.e. \eqref{coe:v3}, and the discriminant of $P_{\mathrm{pr}}$, i.e. \eqref{coe:discriminant} vanishes.
\end{itemize}

\subsection*{Mathematica Implementation and Results}

The script below defines all six equations and then systematically tests the three systems.

\begin{lstlisting}[style=mystyle, breaklines=true, caption={System of Equations in Mathematica}]
$eq_1 = c_1^2 + s_1^2 == 1$;
$eq_2 = c_2^2 + s_2^2 == 1$;
$eq_3 = c_1 c_2 (1 + 2c_1) (1 + 2c_2) - 4s_1^2 s_2^2 == 0$;
$eq_4 = -\frac{8 s_1^4 s_2^3}{3 c_1^3 (1 + 2c_2)^3} - \frac{16 s_1^3 s_2^4}{3 c_1^2 (1 + 2c_2)^3} - \frac{8 c_2 s_1^3 s_2^2}{c_1^2 (1 + 2c_2)^2} 
      - \frac{4 s_1 s_2^2}{1 + 2c_2} - \frac{2 c_2 s_1}{3} - \frac{s_2}{3}  == 0 $;
$eq_5 = -\frac{4 s_1^3 s_2^2}{c_1^2 (1 + 2c_2)^2} - \frac{8 s_1^2 s_2^3}{c_1 (1 + 2c_2)^2} - \frac{8 c_2 s_1^2 s_2}{c_1 (1 + 2c_2)} - 2 c_1 s_2 == 0 $;
$eq_6 = (-\frac{4 s_1^3 s_2^2}{c_1^2 (1 + 2c_2)^2} - \frac{8 s_1^2 s_2^3}{c_1 (1 + 2c_2)^2} - \frac{8 c_2 s_1^2 s_2}{c_1 (1 + 2c_2)} - 2 c_1 s_2)^2 - 4(2c_1c_2+c_1)*$
$ 
\ \ \ \ \ \ (\frac{16 s_1^4 s_2^4}{3 c_1^3 (1 + 2c_2)^3} - \frac{8 c_2 s_1^4 s_2^4}{3 c_1^3 (1 + 2c_2)^4} - \frac{4 s_1^4 s_2^4}{3 c_1^3 (1 + 2c_2)^4}
      + \frac{4 s_1^2 s_2^2}{3 c_1 (1 + 2c_2)} - \frac{4 c_2 s_1^2 s_2^2}{c_1 (1 + 2c_2)^2} - \frac{c_1 c_2}{6} - \frac{c_2}{12}) == 0$;

system1 = {$eq_1$, $eq_2$, $eq_3$, $eq_4$, $c_1 \neq 0$, $c_2 \neq 0$, $c_1 \neq -1/2$, $c_2 \neq -1/2$};
system2 = {$eq_1$, $eq_2$, $eq_3$, $eq_4$, $eq_5$, $c_1 \neq 0$, $c_2 \neq 0$, $c_1 \neq -1/2$, $c_2 \neq -1/2$};
system3 = {$eq_1$, $eq_2$, $eq_3$, $eq_4$, $eq_6$, $c_1 \neq 0$, $c_2 \neq 0$, $c_1 \neq -1/2$, $c_2 \neq -1/2$};

solutions1 = Reduce[system1, {$c_1, s_1, c_2, s_2$}, Reals]
solutions2 = Reduce[system2, {$c_1, s_1, c_2, s_2$}, Reals]
solutions3 = Reduce[system3, {$c_1, s_1, c_2, s_2$}, Reals]
\end{lstlisting}

\subsection*{Interpretation of Results}

The script produces the following outputs for each of the three tests:

\paragraph{Result for System 1}
The code finds a valid instance, confirming that the baseline system is consistent.
\begin{verbatim}
A solution exists for System 1. An example is: 
{{c1 -> -0.996..., s1 -> -0.0869..., c2 -> 0.0288..., s2 -> 1.00...}}
\end{verbatim}

\paragraph{Result for System 2}
The \texttt{Reduce} command returns an empty list, triggering the corresponding output.
\begin{verbatim}
No solution was found for System 2 in the real domain.
\end{verbatim}

\paragraph{Result for System 3}
No solution instance can be found.
\begin{verbatim}
No solution was found for System 3 in the real domain.
\end{verbatim}

\section*{Acknowledgement}
  We are grateful to Professors Ikromov Isroil and Liu Shiping for their helpful discussions.

	\bibliographystyle{amsplain}

\end{document}